\documentclass[10pt]{article}
\usepackage[latin1]{inputenc}
\usepackage[T1]{fontenc}
\usepackage{amsmath,amssymb,amstext}
\usepackage{theorem}
\usepackage{a4wide}
\usepackage{multicol}
\usepackage[english]{babel}
\usepackage{enumerate}

\def\R{{ \mathbb{R}}}
\newtheorem{theorem}{Theorem}

\newtheorem{lemma}[theorem]{Lemma}

\newenvironment{proof}[1][Proof]{\textbf{#1.} }{\ \rule{0.5em}{0.5em}}

\begin{document}
\renewcommand{\thefootnote}{\fnsymbol{footnote}}

\date{February 25, 2011}
\title{\small\bf{{ PARAMETER ESTIMATION FOR  FRACTIONAL ORNSTEIN-UHLENBECK PROCESSES: NON-ERGODIC CASE}}}
\author{ \small {Rachid Belfadli$^{1}$ $\qquad$ Khalifa Es-Sebaiy$^{2}$$\qquad$ Youssef Ouknine$^{3}$$\qquad$\vspace{2mm}} \\
 \small$^1$ Polydisciplinary Faculty of Taroudant, University Ibn Zohr, Taroudant, Morocco.\\
 \small belfadli@gmail.com\vspace{2mm}\\
\small$^2$ Institut de Mathématiques de Bourgogne, Université de Bourgogne, Dijon, France.\\
\small khalifasbai@gmail.com\vspace{2mm}\\
\small$^3$ Department of Mathematics, Faculty of Sciences Semlalia, Cadi Ayyad University, \\\small 2390 Marrakesh, Morocco.\\\small ouknine@ucam.ac.ma
 \vspace*{0.1in}}

\maketitle

\begin{abstract}
We consider the parameter estimation problem for the non-ergodic fractional Ornstein-Uhlenbeck process defined as $dX_t=\theta X_tdt+dB_t,\  t\geq0$, with a parameter  $\theta>0$, where $B$ is  a fractional
Brownian motion of Hurst index $H\in(\frac{1}{2},1)$. We
study the consistency and the asymptotic distributions of the least squares estimator $\widehat{\theta}_t$ of  $\theta$ based on the observation $\{X_s,\ s\in[0,t]\}$ as $t\rightarrow\infty$.
\end{abstract}

\vskip0.5cm


\vskip0.2cm


\renewcommand{\thefootnote}{\fnsymbol{footnote}}

\vskip1cm

\section{Introduction}
We consider the Ornstein-Uhlenbeck process $X=\left\{X_t, t\geq0\right\}$ given by the following linear stochastic differential equation
\begin{eqnarray}\label{OU}X_0=0;\quad  dX_t=\theta X_tdt+dB_t,\quad t\geq0,
\end{eqnarray}where  $B$ is a fractional Brownian motion of Hurst index $H>\frac{1}{2}$ and  $\theta\in(-\infty,\infty) $   is an unknown parameter. An interesting problem is to estimate the parameter $\theta$ when one observes the whole trajectory of $X$.
 First, let us recall some results in the case  when $B$ is a standard Brownian motion.
  In this special case, the parameter estimation for $\theta$ has been well studied by using the classical maximum likelihood method or by using the trajectory fitting method. If $\theta<0$ (ergodic case), the maximum likelihood estimator (MLE) of $\theta$ is asymptotically normal (see Liptser and Shiryaev \cite{LS}, Kutoyants \cite{Kutoyants}). If $\theta>0$ (non-ergodic case), the MLE of $\theta$ is asymptotically Cauchy (see  Basawa and  Scott \cite{BS}, Dietz and Kutoyants \cite{DK}). Recently, in a more general context, several authors extended this study  to  some generalizations of Ornstein-Uhlenbeck process  driven by Brownian motion (for instance, Barczy and Pap \cite{BP}). Similar properties of the asymptotic behaviour of MLE has also been obtained  with respect to the  trajectory fitting estimators (see Dietz and Kutoyants \cite{DK}).

When $B$ is replaced by an $\alpha$-stable Lévy motion in the equation (\ref{OU}), Hu and Long \cite{HL} discussed the parameter estimation of $\theta$ in both the ergodic and the non-ergodic cases. They used the trajectory fitting method combined with the weighted least squares technique.\vspace{2mm}

Now, let  us consider a parameter estimation problem of the parameter $\theta$ for the fractional Ornstein-Uhlenbeck process  $X$ of (\ref{OU}).

In the case  $\theta<0$ (corresponding to the ergodic case),  Hu and Nualart \cite{HN} studied  the parameter estimation for $\theta$  by using the least squares estimator (LSE) defined as
\begin{eqnarray}\label{(LSE)} \widehat{\theta}_t =\frac{\int_0^tX_sdX_s}{\int_0^tX_s^2ds},\quad t\geq0.
\end{eqnarray}This LSE is obtained by the least squares technique, that is, $ \widehat{\theta}_t$ (formally) minimizes
\[\theta \longmapsto \int_0^t\left|\dot{{X}}_s+\theta X_s\right|^2ds.\]
To obtain the consistency of the LSE $\widehat{\theta}_t$, the authors of  \cite{HN}  are forced to consider $\int_0^tX_sdX_s$ as a Skorohod integral rather than Young integral
in the definition (\ref{(LSE)}). Assuming $\int_0^tX_sdX_s$ is a Skorohod integral and $\theta<0$, they proved the strong consistence of $\widehat{\theta}_t$ if $H\geq\frac{1}{2}$, and  that the LSE $\widehat{\theta}_t$ of $\theta$ is asymptotically normal if $H\in[\frac{1}{2},\frac{3}{4})$. Their proof of the central limit theorem is based on the fourth moment theorem of Nualart and Peccati \cite{NP}.

In this paper, our purpose is to study the non-ergodic case corresponding to $\theta>0$. More precisely, we shall estimate $\theta$ by  the LSE $\widehat{\theta}_t$ defined in  (\ref{(LSE)}), where in our case,  the   integral $\int_0^tX_sdX_s$ is interpreted as a  Young integral. Indeed in that case, we have $\widehat{\theta}_t =\frac{X_t^2}{2\int_0^tX_s^2ds}$ which converges almost surely to $\theta$, as $t$ tends to infinity (see   Theorem \ref{convergence almost surely}). Moreover, it turned out that the path-wise approach is the preferred way to simulate numerically  an estimator $\widehat{\theta}_t$. Our technics used in this work are  inspired from the recent
paper by Es-Sebaiy and Nourdin \cite{EN}.

The organization of our paper is as follows. Section 2 contains the presentation of the
basic tools that we will need throughout the paper: fractional Brownian motion, Malliavin derivative,  Skorohod integral, Young integral and the link between Young and Skorohod integrals.
The aim of Section 3 is twofold. Firstly, we prove when $H>\frac{1}{2}$ the strong consistence of the LSE $\widehat{\theta}_t$, that is, $\widehat{\theta}_t$ converges almost surely to $\theta$, as $t$ goes to infinity. Secondly,  we investigate the asymptotic distribution of our estimator $\widehat{\theta}_t$ in the case $H>\frac{1}{2}$.  We  obtain that (see Theorem \ref{convergence distribution})
\begin{eqnarray*}
e^{\theta t}\left(\widehat{\theta}_t-\theta\right) \overset{\texttt{law}}{\longrightarrow}2\theta\mathcal{C}(1)
\quad \mbox{ as } t\longrightarrow \infty,
\end{eqnarray*}with  $\mathcal{C}(1)$  the standard Cauchy distribution  with the probability density function
$\frac{1}{\pi (1+x^2)};\  x\in \mathbb{R}$.
\section{Preliminaries}
In this section we describe some basic facts on the  stochastic calculus with respect to a
fractional Brownian motion. For more complete presentation on the subject, see   \cite{Nualart}, \cite{AN} and \cite{Nourdin}.\\
The  fractional Brownian motion $(B_t, t\geq0)$ with Hurst parameter $H\in(0,1)$, is defined as a centered Gaussian process starting from zero with covariance
\[R_H(t,s)=E(B_tB_s)=\frac{1}{2}\left(t^{2H}+s^{2H}-|t-s|^{2H}\right).\]
We assume that $B$ is defined on a complete probability space $(\Omega, \mathcal{F}, P)$ such
that $\mathcal{F}$ is the sigma-field generated by $B$. By  Kolmogorov's continuity criterion and the fact  $$E\left(B_t-B_s\right)^2=|s-t|^{2H};\ s,\ t\geq~0,$$ we deduce that $B$ has H\"older continuous paths of order $H-\varepsilon$, for all $\varepsilon\in(0,H)$.

 Fix a time interval $[0, T]$. We denote by $\cal{H}$ the canonical Hilbert space associated to the  fractional
Brownian motion $B$. That is, $\cal{H}$ is the closure of the linear span $\mathcal{E}$ generated by the indicator functions $ 1_{[0,t]},\ t\in[0,T] $ with respect to the scalar product \[\langle1_{[0,t]},1_{[0,s]}\rangle=R_H(t,s).\]
The application $\varphi\in\mathcal{E}\longrightarrow B(\varphi)$ is an isometry from $\mathcal{E}$ to the Gaussian space generated by $B$ and it can be extended to $\cal{H}$.\\
If $H>\frac{1}{2}$ the elements of $\cal{H}$ may be not functions but distributions of negative order (see \cite{PT}). Therefore, it is of interest to know significant subspaces of functions contained in it.\\
Let $|\cal{H}|$  be the set of measurable functions  $\varphi$   on $[0, T]$ such that
\[\|\varphi\|_{|\cal{H}|}^2:=H(2H-1)\int_0^T\int_0^T|\varphi(u)||\varphi(v)||u-v|^{2H-2}dudv<\infty.\] Note that, if $\varphi,\ \psi\in|\cal{H}|$,
\[ E(B(\varphi)B(\psi))=H(2H-1)\int_0^T\int_0^T\varphi(u)\psi(v)|u-v|^{2H-2}dudv.\]
It follows actually from \cite{PT}  that the space $|\cal{H}|$ is a Banach space for the norm $\|.\|_{|\cal{H}|}$ and it is included in $\cal{H}$. In fact,
\[L^2([0, T]) \subset L^{\frac{1}{H}}([0, T])\subset|\cal{H}|\subset\cal{H}.\]

Let
$\mathrm{C}_b^{\infty}(\R^n,\R)$
 be the class of infinitely
differentiable functions $f: \R^n \longrightarrow \R$ such that $f$ and all its partial derivatives are bounded.
We denote by $\cal{S}$ the class of smooth cylindrical random
variables F of the form
\begin{eqnarray}F = f(B(\varphi_1),...,B(\varphi_n)),\label{functional}\end{eqnarray} where $n\geq1$, $f\in \mathrm{C}_b^{\infty}(\R^n,\R)$
 and $\varphi_1,...,\varphi_n\in\cal{H}.$\\
 The derivative operator $D$ of a smooth and cylindrical random variable $F$ of the form (\ref{functional}) is
defined as the $\cal{H}$-valued random variable
$$D_tF=\sum_{i=1}^{N}\frac{\partial f}{\partial x_i}(B(\varphi_1),...,B(\varphi_n))\varphi_i(t)$$
In this way the derivative $DF$ is an element of $L^2(\Omega
;\cal{H})$. We denote by $D^{1,2}$ the closure of $\mathcal{S}$ with
respect to the norm defined by
$$\|F\|_{1,2}^2=E(\|F\|^2)+E(\|DF\|^2_{{\cal{H}}}).$$
The divergence operator $\delta$ is the adjoint of the derivative
operator $D$. Concretely, a random
variable $u\in L^2(\Omega;\cal{H})$ belongs to the domain of the divergence operator $Dom\delta$ if
\[E\left|\langle DF,u\rangle_{\cal{H}}\right|\leq c\|F\|_{L^2(\Omega)}\]for every $F\in \mathcal{S}$. In this case $\delta(u)$ is given by the duality relationship
\begin{eqnarray*}E(F\delta(u))=E\left<DF,u\right>_{\cal{H}}
\end{eqnarray*}
for any $F\in D^{1,2}$. We will
make use of the notation
$$\delta(u)=\int_0^Tu_s\delta B_s,\quad u\in Dom(\delta).$$
In particular, for $h\in\cal{H}$, $B(h)=\delta(h)=\int_0^Th_s\delta B_s.$

For every $n\geq1$, let ${\cal{H}}_n$ be the nth Wiener chaos of $B$
, that is, the closed linear subspace of $L^2(\Omega)$ generated by the random variables $\{H_n(B(h)), h\in{\cal{H}}, \|h\|_{{\cal{H}}} = 1\}$ where $H_n$ is the nth Hermite polynomial. The mapping $I_n(h^{\otimes n})=n!H_n(B(h))$ provides a linear isometry between the symmetric tensor product ${\cal{H}}^{\odot n}$ (equipped with the modified norm $\|.\|_{{\cal{H}}^{\odot n}}=\frac{1}{\sqrt{n!}}\|.\|_{{\cal{H}}^{\otimes n}}$) and ${\cal{H}}_n$. For every  $f,g\in{{\cal{H}}}^{\odot n}$ the following multiplication formula holds
\[E\left(I_n(f)I_n(g)\right)=n!\langle f,g\rangle_{{\cal{H}}^{\otimes n}}.\]
 Finally, It is well-known   that $L^2(\Omega)$ can be decomposed into the infinite orthogonal sum of the spaces ${\cal{H}}_n$. That is, any square
integrable random variable $F\in L^2(\Omega)$  admits the following chaotic expansion
\[F=E(F)+\sum_{n=1}^{\infty}I_n(f_n),\]
where the $f_n \in{{\cal{H}}}^{\odot n}$  are uniquely determined by
$F$.

Fix $T>0$. Let $f, g: [0,T]\longrightarrow\mathbb{R}$ are H\"older continuous functions of orders $\alpha\in(0,1)$
and $\beta\in(0,1)$ with $\alpha+\beta>1$. Young \cite{Young} proved that the Riemann-Stieltjes
integral (so-called Young integral) $\int_0^Tf_sdg_s$ exists. Moreover, if $\alpha=\beta\in(\frac{1}{2},1)$ and
$\phi: \mathbb{R}^2\longrightarrow\mathbb{R}$ is a function of class $\mathcal{C}^1$,
  the integrals $\int_0^.\frac{\partial\phi}{\partial f}(f_u,g_u)df_u$ and $\int_0^.\frac{\partial\phi}{\partial g}(f_u,g_u)dg_u$ exist in the Young sense and the following change of variables formula holds:
\begin{eqnarray}\label{change of variables formula}
\phi(f_t,g_t)=\phi(f_0,g_0)+\int_0^t\frac{\partial\phi}{\partial f}(f_u,g_u)df_u+\int_0^t\frac{\partial\phi}{\partial g}(f_u,g_u)dg_u,\quad 0\leq t\leq T.
\end{eqnarray}
As a consequence, if $H>\frac{1}{2}$ and $(u_t,\ t\in[0, T])$ be a process with H\"older paths of order $\alpha>1-H$, the integral $\int_0^Tu_sdB_s$ is well-defined as Young integral. Suppose moreover that for any  $ t\in[0,T]$, $u_t\in D^{1,2}$, and
\[P\left(\int_0^T\int_0^T|D_su_t||t-s|^{2H-2}dsdt<\infty\right)=1.\]
Then, by \cite{AN}, $u\in Dom\delta$ and for every $ t\in[0,T]$,
\begin{eqnarray}\label{link young skorohod}\int_0^tu_sdB_s=\int_0^tu_s\delta B_s+H(2H-1)\int_0^t\int_0^tD_su_r|s-r|^{2H-2}drds.
\end{eqnarray}
In particular, when $\varphi$ is a non-random  H\"older continuous function of order $\alpha>1-H$, we obtain
\begin{eqnarray}\label{non random}\int_0^T\varphi_sdB_s=\int_0^T\varphi_s\delta B_s=B(\varphi).\end{eqnarray}
\\In addition,  for all $\varphi,\ \psi\in|\cal{H}|$,
\begin{eqnarray}E\left(\int_0^T\varphi_sdB_s\int_0^T\psi_sdB_s\right)
=H(2H-1)\int_0^T\int_0^T\varphi(u)\psi(v)|u-v|^{2H-2}dudv.\end{eqnarray}

\section{Asymptotic behavior of the least squares estimator}
Throughout this paper we assume $H\in(\frac{1}{2},1)$ and $\theta>0$.
Let us consider the equation (\ref{OU}) driven by a fractional Brownian motion $B$  with Hurst parameter $H$ and
$\theta$ is the unknown parameter to be estimated from the observation $X$. The linear equation (\ref{OU}) has the
following explicit solution:
\begin{eqnarray}\label{explicit solution}X_t=e^{\theta t}\int_0^te^{-\theta s}dB_s,\qquad t\geq0,
\end{eqnarray}
where the integral $\int_0^te^{-\theta s}dB_s$ is a  Young integral. \\
Let us introduce the following process
\[\xi_t:=\int_0^te^{-\theta s}dB_s,\qquad t\geq0.\]By using the equation (\ref{OU}) and (\ref{explicit solution}) we can write the LSE $\widehat{\theta}_t$ defined in (\ref{(LSE)}) as follows
\begin{eqnarray}\label{representation of LSE}\widehat{\theta}_t=\theta+\frac{\int_0^tX_sdB_s}{\int_0^tX_s^2ds}=\theta+\frac{\int_0^te^{\theta s}\xi_sdB_s}{\int_0^te^{2\theta s}\xi_s^2ds}.
\end{eqnarray}
\subsection{Consistency of the estimator LSE}
The following theorem proves the strong consistency of the LSE $\widehat{\theta}_t $.
\begin{theorem}\label{convergence almost surely} Assume $H\in(\frac{1}{2},1)$, then
\begin{eqnarray*}\widehat{\theta}_t   {\longrightarrow} \theta\ \mbox{ almost surely}
\end{eqnarray*}as $t\longrightarrow\infty.$
\end{theorem}

For the proof of Theorem \ref{convergence almost surely} we need the following two lemmas.
\begin{lemma}\label{convergence p.s of xi}Suppose that  $H>\frac{1}{2}$. Then
\begin{itemize}
\item[i)] For all $\varepsilon\in(0,H)$, the process $\xi$  admits a modification with $(H-\varepsilon)$-H\"older continuous paths, still denoted $\xi$ in the sequel.
\item[ii)] $\xi_t\longrightarrow \xi_{\infty}:=\int_0^{\infty}e^{-\theta r}dB_r$ almost surely and in $L^2(\Omega)$ as $t\longrightarrow\infty$.
\end{itemize}
\end{lemma}
\begin{lemma}\label{convergence of dominating}
Let $H>\frac{1}{2}$. Then, as $t\rightarrow \infty$, \[e^{-2\theta t}\int_0^tX_s^2ds=e^{-2\theta t}\int_0^te^{2\theta s}\xi_s^2ds\longrightarrow \frac{\xi_{\infty}^2}{2\theta}\mbox{ almost surely.   }\]
\end{lemma}
\begin{proof}[Proof of Lemma \ref{convergence p.s of xi}] We prove the point $i)$. We have, for every $0\leq s<t$,
\begin{eqnarray*}E\left(\xi_t-\xi_s\right)^2&=&E\left(\int_s^te^{-\theta r}dB_r\right)^2\\
&=&H(2H-1)\int_s^t\int_s^te^{-\theta u}e^{-\theta v}|u-v|^{2H-2}dudv\\
&\leq&H(2H-1)\int_s^t\int_s^t|u-v|^{2H-2}dudv\\
&=&E\left(B_t-B_s\right)^2=|t-s|^{2H}.
\end{eqnarray*}Thus, by applying  the Kolmogorov-Centsov theorem  to the centered gaussian process $\xi$ we deduce $i)$.
\\ Concerning the second point $ii)$, we first notice that the integral $\xi_{\infty}=\int_0^{\infty}e^{-\theta r}dB_r$ is well defined. In fact,
\begin{eqnarray}&&H(2H-1)\int_0^{\infty}\int_0^{\infty}e^{-\theta r}e^{-\theta s}|r-s|^{2H-2}drds\nonumber\\
&=&2H(2H-1)\int_0^{\infty}ds e^{-\theta s}\int_0^sdr e^{-\theta r}(s-r)^{2H-2}\nonumber\\
&=&2H(2H-1)\int_0^{\infty}ds e^{-2\theta s}\int_0^s due^{ \theta u}u^{2H-2}\nonumber\\
&=&2H(2H-1)\int_0^{\infty}due^{ \theta u}u^{2H-2}\int_u^{\infty} ds e^{-2\theta s}\nonumber\\
&=&\frac{H(2H-1)}{\theta}\int_0^{\infty}e^{ -\theta u}u^{2H-2}du\nonumber
\\&=& \frac{H(2H-1)}{\theta^{2H}}\Gamma(2H-1)=\frac{H\Gamma(2H)}{\theta^{2H}}<\infty,\label{xi_t-xi_{infty}}
\end{eqnarray}with $\Gamma$ denotes the classical Gamma function.
Moreover, $\xi_t$ converges to $\xi_{\infty}$ in $L^2(\Omega)$. Indeed,
\begin{eqnarray*}E\left[(\xi_{t}-\xi_{\infty})^2\right]&=&H(2H-1)\int_t^{\infty}\int_t^{\infty}e^{-\theta r}e^{-\theta s}|r-s|^{2H-2}drds\nonumber\\
&=&2H(2H-1)\int_t^{\infty}ds e^{-\theta s}\int_t^sdr e^{-\theta r}(s-r)^{2H-2}\nonumber\\
&=&2H(2H-1)\int_t^{\infty}dv e^{-2\theta s}\int_0^{s-t} due^{ \theta u}u^{2H-2}\nonumber\\
&=&2H(2H-1)e^{-2\theta t}\int_0^{\infty}dv e^{-2\theta v}\int_0^{v} due^{ \theta u}u^{2H-2}\nonumber\\
&=&2H(2H-1)e^{-2\theta t}\int_0^{\infty}due^{ \theta u}u^{2H-2}\int_u^{\infty} dv e^{-2\theta v}\nonumber\\
&=&\frac{H(2H-1)}{\theta}e^{-2\theta t}\int_0^{\infty}e^{ -\theta u}u^{2H-2}du\nonumber
\\&=& \frac{H\Gamma(2H)}{\theta^{2H}}e^{-2\theta t}\\&\rightarrow& 0 \mbox{ as } t\rightarrow \infty.\nonumber
\end{eqnarray*}
Now, let us show that $\xi_t\longrightarrow \xi_{\infty}$ almost surely as $t\rightarrow \infty$.
By using Borel-Cantelli lemma, it is sufficient to prove that, for any $\varepsilon > 0$
\begin{eqnarray}\label{BC}
\sum_{n\geq0} P\left(\sup_{n\leq t \leq n+1}\left|\int_{t}^{\infty}e^{-\theta s}dB_s\right|>\varepsilon\right)
<\infty.
\end{eqnarray}
For this purpose, let $\frac{1}{2}<\alpha<H$. As in the proof of [Theorem 4, \cite{AN}], we can write for every $t>0$
\begin{eqnarray*}
\int_{t}^{\infty}e^{-\theta s}dB_s= c_{\alpha}^{-1} \int_{t}^{\infty}dB_s e^{-\theta s}\left(\int_{t}^{s}dr(s-r)^{-\alpha}(r-t)^{\alpha-1}\right),
\end{eqnarray*}
with $c_{\alpha}= \int_{t}^{s}(s-r)^{-\alpha}(r-t)^{\alpha-1}dr= \beta(\alpha, 1- \alpha)$, where $\beta$ is the Beta function.\\
By Fubini's stochastic theorem (see for example \cite{Nualart}), we have
\begin{eqnarray*}
\int_{t}^{\infty}e^{-\theta s}dB_s= c_{\alpha}^{-1} \int_{t}^{\infty}dr(r-t)^{\alpha-1}\left(\int_{r}^{\infty}dB_s e^{-\theta s}(s-r)^{-\alpha}\right).
\end{eqnarray*}
Cauchy-Schwarz's inequality implies that
\begin{eqnarray*}
&&\left|\int_{t}^{\infty}e^{-\theta s}dB_s\right|^{2}\\&\leq& c_{\alpha}^{-2}\left(\int_{t}^{\infty}(r-t)^{2(\alpha-1)}
e^{-\theta(r-t)}dr\right)\left(\int_{t}^{\infty}dr e^{-\theta(r-t)}\left|\int_{r}^{\infty}dB_s e^{-\theta s}(s-r)^{-\alpha}e^{\theta (r-t)}\right|^{2}\right)\\
&=& \frac{c_{\alpha}^{-2} \Gamma(2 \alpha -1)}{\theta ^{2 \alpha -1}} e^{-2 \theta t}\int_{t}^{\infty}dr e^{-\theta(r-t)}\left|\int_{r}^{\infty}dB_s(s-r)^{-\alpha}e^{-\theta (s-r)}\right|^{2}
\end{eqnarray*}
Thus,
\begin{eqnarray*}
&&\sup_{n\leq t \leq {n+1}}\left|\int_{t}^{\infty}e^{-\theta s}dB_s\right|^{2}\\&&\leq\frac{c_{\alpha}^{-2} \Gamma(2 \alpha -1)}{\theta ^{2 \alpha -1}}e^{-2 \theta n}e^{\theta}\int_{n}^{\infty}dr e^{-\theta(r-n)}\left|\int_{r}^{\infty}dB_s(s-r)^{-\alpha}e^{-\theta (s-r)}\right|^{2}
\end{eqnarray*}
On the other hand,
\begin{eqnarray*}
 && E \left( \left|\int_{r}^{\infty}(s-r)^{-\alpha}e^{-\theta (s-r)}dB_s\right|^{2}\right)\\&=&H(2H-1)
  \int_{r}^{\infty}dv(v-r)^{-\alpha}e^{-\theta (v-r)}\int_{r}^{\infty}du(u-r)^{-\alpha}e^{-\theta (u-r)}|u-v|^{2H-2}\\&=&H(2H-1)
  \int_{0}^{\infty}dvv^{-\alpha}e^{-\theta v}\int_{0}^{\infty}duu^{-\alpha}e^{-\theta u}|u-v|^{2H-2}
  \\&=&2H(2H-1)
  \int_{0}^{\infty}dvv^{-\alpha}e^{-\theta v}\int_{0}^{v}duu^{-\alpha}e^{-\theta u}(v-u)^{2H-2}\\&=&2H(2H-1)
  \int_{0}^{\infty}dvv^{-\alpha}e^{-\theta v}\int_{0}^{v}du(v-u)^{-\alpha}e^{-\theta(v- u)}u^{2H-2}\\&\leq&2H(2H-1)
  \int_{0}^{\infty}dvv^{-\alpha}e^{-\theta v}\int_{0}^{v}du(v-u)^{-\alpha}u^{2H-2}\\&=&2H(2H-1)
  \int_{0}^{\infty}dvv^{2H-2\alpha-1}e^{-\theta v}\int_{0}^{1}du u^{2H-2}(1-u)^{-\alpha}\\&=&2H(2H-1)
  \frac{\Gamma(2H-2\alpha)\beta(2H-1,1-\alpha)}{\theta^{2H-2\alpha}}:=C_1(\alpha,H,\theta)<\infty.
\end{eqnarray*}
 Combining this with the fact that $\int_{n}^{\infty} e^{-\theta(r-n)}dr=\frac{1}{\theta}$, we obtain
\begin{eqnarray*}
E\left(\sup_{n\leq t \leq {n+1}}\left|\int_{t}^{\infty}e^{-\theta s}dB_s\right|^{2}\right)&\leq& C_2(\alpha, H, \theta)e^{-2\theta n},
\end{eqnarray*}
with
\[C_2(\alpha, H, \theta)= \frac{c_{\alpha}^{-2} \Gamma(2 \alpha -1)e^{\theta}}{\theta ^{2 \alpha}}C_1(\alpha,H,\theta).\]
Consequently,
\begin{eqnarray*}
\sum_{n\geq0} P\left(\sup_{n\leq t \leq {n+1}}\left|\int_{t}^{\infty}e^{-\theta s}dB_s\right|>\varepsilon\right)
&\leq& \varepsilon^{-2}\sum_{n\geq0}E\left(\sup_{n\leq t \leq {n+1}}\left|\int_{t}^{\infty}e^{-\theta s}dB_s\right|^2\right)\\
&\leq&\varepsilon^{-2}C_2(\alpha, H, \theta)\sum_{n\geq0} e^{-2\theta n} <\infty.
\end{eqnarray*}
This finishes the proof of the claim (\ref{BC}), and thus the proof of Lemma \ref{convergence p.s of xi}.
\end{proof}\\
\begin{proof}[Proof of Lemma \ref{convergence of dominating}]
Using (\ref{xi_t-xi_{infty}}),  we have  $$E[\xi_{\infty}^2]=\frac{H\Gamma(2H)}{\theta^{2H}}<\infty.$$ Hence $\xi_{\infty}\backsim \mathcal{N}(0,\frac{H\Gamma(2H)}{\theta^{2H}})$ and this implies that
\begin{eqnarray}\label{xi_infty<fini}P(\xi_{\infty}=0)=0.
\end{eqnarray}
The continuity of $\xi$ entails that, for every $t>0$    \begin{eqnarray}\label{minoration xi}\int_0^te^{2\theta s}\xi_s^2ds\geq
 \int_{\frac{t}{2}}^te^{2\theta s}\xi_s^2ds\geq\frac{t}{2}e^{\theta t}\left(\inf_{\frac{t}{2}\leq s\leq t}\xi_{s}^2\right)\ \mbox{ almost surely}.\end{eqnarray} Furthermore, the continuity of $\xi$ and the point ii) of Lemma \ref{convergence p.s of xi} yield  \[\lim_{t\rightarrow\infty}\left(\inf_{\frac{t}{2}\leq s\leq t}\xi_{s}^2\right)=\xi_{\infty}\ \mbox{ almost surely}.\]Combining  this last convergence with (\ref{minoration xi}) and (\ref{xi_infty<fini}), we deduce that
  \[\lim_{t\rightarrow\infty}\int_0^te^{2\theta s}\xi_s^2ds=\infty\ \mbox{ almost surely}.\]
Hence, we can use L'Hôspital's rule and we obtain
\begin{eqnarray*}\lim_{t\rightarrow\infty}\frac{\int_0^te^{2\theta s}\xi_s^2ds}{e^{2\theta t}} = \lim_{t\rightarrow\infty}\frac{\xi_t^2}{2\theta}
 = \frac{\xi_{\infty}^2}{2\theta}\quad \mbox{ almost surely.}
\end{eqnarray*}This completes the proof of Lemma \ref{convergence of dominating}.\end{proof} \\
\begin{proof}[Proof of Theorem \ref{convergence almost surely}]
Using the change of variable formula (\ref{change of variables formula}), we conclude that
\begin{eqnarray*}\frac{1}{2}e^{2\theta t}\xi_t^2=\theta\int_0^te^{2\theta s}\xi_s^2ds+\int_0^te^{ \theta s}\xi_s dB_s
\end{eqnarray*}Hence
\begin{eqnarray*}\widehat{\theta}_t-\theta =\frac{\int_0^te^{\theta s}\xi_sdB_s}{\int_0^te^{2\theta s}\xi_s^2ds}=\frac{\xi_t^2}{2e^{-2\theta t}\int_0^te^{2\theta s}\xi_s^2ds}-\theta.
\end{eqnarray*}Combining this with Lemma \ref{convergence p.s of xi} and Lemma \ref{convergence of dominating}, we deduce  that $\widehat{\theta}_t\rightarrow\theta$ almost surely as $t\longrightarrow\infty$.
\end{proof}

\subsection{Asymptotic distribution of the estimator LSE}
This paragraph is devoted to the investigation of asymptotic distribution of the LSE $\widehat{\theta}_t$ of $\theta$.
We start with the following lemma.
\begin{lemma}\label{decomposition Young Skorohod}
Suppose that $H>\frac{1}{2}$. Then, for every $t\geq0$, we have \begin{eqnarray*}\int_0^tdB_se^{\theta s}\int_0^sdB_re^{-\theta r}&=&\int_0^tdB_se^{\theta s}\int_0^tdB_re^{-\theta r}-\int_0^t\delta B_se^{-\theta s}\int_0^s\delta B_re^{\theta r}\\&&-H(2H-1)\int_0^tdse^{-\theta s}\int_0^sdre^{\theta r}|s-r|^{2H-2}.
\end{eqnarray*}\end{lemma}
\begin{proof}
Let $t\geq0$. By  the change of variable formula (\ref{change of variables formula})
\begin{eqnarray*}\int_0^tdB_se^{\theta s}\int_0^sdB_re^{-\theta r}&=&\int_0^tdB_se^{\theta s}\int_0^tdB_re^{-\theta r}-\int_0^tdB_se^{-\theta s}\int_0^sdB_re^{\theta r}.
\end{eqnarray*}
On the other hand, according to (\ref{link young skorohod}) and (\ref{non random}),
\begin{eqnarray*}&&\int_0^t dB_se^{-\theta s}\int_0^sdB_re^{\theta r}\\&=&\int_0^t\delta B_se^{-\theta s}\int_0^s\delta B_re^{\theta r}+H(2H-1)\int_0^tdse^{-\theta s}\int_0^sdre^{\theta r}|s-r|^{2H-2},
\end{eqnarray*} which completes the proof.
\end{proof}
\begin{theorem}\label{convergence distribution}Let $H>\frac{1}{2}$ be fixed. Then, as $t\longrightarrow \infty$,
\begin{eqnarray*}e^{\theta t}\left(\widehat{\theta}_t-\theta\right)\overset{\texttt{law}}{\longrightarrow}2\theta\mathcal{C}(1),
\end{eqnarray*}with $\mathcal{C}(1)$  the standard Cauchy distribution.
\end{theorem}

In order to prove Theorem \ref{convergence distribution} we need the following two lemmas.
\begin{lemma}\label{lemma in law}Fix $H>\frac{1}{2}$. Let $F$ be any $\sigma\{B\}$-measurable random variable such that $P(F<\infty)=1$. Then, as $t\longrightarrow\infty$,  \begin{eqnarray*}\left(F, e^{-\theta t}\int_0^te^{\theta s}dB_s\right)\overset{\texttt{law}}{\longrightarrow}\left(F, \frac{\sqrt{H\Gamma(2H)}}{\theta^{H}}N\right),
\end{eqnarray*}where $N\sim \mathcal{N}(0,1)$ is independent of $B$.
\end{lemma}
\begin{lemma}\label{lemma in l^2} Let $H>\frac{1}{2}$. Then, as $ t\rightarrow \infty$,
\begin{eqnarray}\label{1}e^{-\frac{\theta t}{2}}\int_0^t\delta B_se^{-\theta s}\int_0^s\delta B_re^{\theta r}\longrightarrow0\ \mbox{ in }L^2(\Omega),\end{eqnarray}and
\begin{eqnarray}\label{2}e^{-\frac{\theta t}{2}}\int_0^tdse^{-\theta s}\int_0^sdre^{\theta r}|s-r|^{2H-2}\longrightarrow0.
\end{eqnarray}
\end{lemma}
\begin{proof}[Proof of Lemma \ref{lemma in law}]
For any $d\geq1$, $s_1\ldots s_d\in[0,\infty)$, we shall prove that, as $t\longrightarrow\infty$,
\begin{eqnarray}\label{law technic}\left(B_{s_1},\ldots,B_{s_d},e^{-\theta t}\int_0^te^{\theta s}dB_s\right)\overset{\texttt{law}}{\longrightarrow}
\left(B_{s_1},\ldots,B_{s_d},\frac{\sqrt{H\Gamma(2H)}}{\theta^{H}}N\right)
\end{eqnarray} which is enough to lead to the desired conclusion. Because the left-hand side in the previous
convergence is a Gaussian vector (see proof of [Lemma 7, \cite{EN}]), to get (\ref{law technic}) it is sufficient to check the convergence of its
covariance matrix. Let us first compute the limiting variance of $e^{-\theta t}\int_0^te^{\theta s}dB_s$  as $t\longrightarrow\infty$. We have
\begin{eqnarray*}E\left[\left(e^{-\theta t}\int_0^te^{\theta s}dB_s\right)^2\right]&=&H(2H-1)e^{-2\theta t}\int_0^t\int_0^te^{\theta s}e^{\theta r}|s-r|^{2H-2}drds\\
&=&2H(2H-1)e^{-2\theta t}\int_0^tdse^{\theta s}\int_0^s dr e^{\theta r}|s-r|^{2H-2}\\
&=&2H(2H-1)e^{-2\theta t}\int_0^tdse^{2\theta s}\int_0^s dr e^{-\theta r}r^{2H-2}\\
&=&2H(2H-1)e^{-2\theta t}\int_0^t dr e^{-\theta r}r^{2H-2}\int_r^t dse^{2\theta s}\\
&=&\frac{H(2H-1)}{\theta}\left(\int_0^t r^{2H-2}e^{-\theta r}dr-e^{-2\theta t}\int_0^t r^{2H-2} e^{\theta r}dr \right)\\&\rightarrow&\frac{H\Gamma(2H)}{\theta^{2H}}  \  \mbox{ as }t\rightarrow \infty,
\end{eqnarray*}because $e^{-2\theta t}\int_0^t r^{2H-2} e^{\theta r}dr\leq e^{- \theta t}\int_0^t r^{2H-2}dr=\frac{t^{2H-1}}{(2H-1)e^{\theta t}}\rightarrow 0$ as $t\rightarrow \infty$.\\
Thus,
\begin{eqnarray*}\lim_{t\rightarrow\infty}E\left[\left(e^{-\theta t}\int_0^te^{\theta s}dB_s\right)^2\right]=\frac{H\Gamma(2H)}{\theta^{2H}}.
\end{eqnarray*}
Hence, to finish the proof it remains to check that, for all fixed $s\geq0$,
\begin{eqnarray*}\lim_{t\rightarrow\infty}E\left(B_s \times e^{-\theta t}\int_0^te^{\theta v}dB_v\right)=0.
\end{eqnarray*}
Indeed, for $s<t$,
\begin{eqnarray*}&&E\left(B_s \times e^{-\theta t}\int_0^te^{\theta v}dB_v\right)\\&=&H(2H-1)e^{-\theta t}\int_0^tdve^{\theta v}\int_0^sdu |u-v|^{2H-2}\\
&=&H(2H-1)e^{-\theta t}\int_0^sdve^{\theta v}\int_0^sdu |u-v|^{2H-2}+
H(2H-1)e^{-\theta t}\int_s^tdve^{\theta v}\int_0^sdu (v-u)^{2H-2}\\
&=&H(2H-1)e^{-\theta t}\int_0^sdve^{\theta v}\int_0^sdu |u-v|^{2H-2}+
He^{-\theta t}\int_s^te^{\theta v}(v^{2H-1}-(v-s)^{2H-1})dv\\
&:=&I_t+J_t.
\end{eqnarray*}It's clear that $I_t\rightarrow 0 $ as $t\rightarrow \infty$.\\
Using integration by parts, the therm $J_t$ can be written as
\begin{eqnarray*}&&J_t\\&=&He^{-\theta t}\int_s^te^{\theta v}(v^{2H-1}-(v-s)^{2H-1})dv\\
&=&He^{-\theta t}\left(\frac{e^{\theta t}}{\theta}[t^{2H-1}-(t-s)^{2H-1}]-\frac{e^{\theta s}}{\theta}s^{2H-1}
+\frac{2H-1}{\theta}\int_s^te^{\theta v}[(v-s)^{2H-2}-v^{2H-2}]dv\right)\\
&\leq&\frac{H}{\theta}[t^{2H-1}-(t-s)^{2H-1}]+\frac{H(2H-1)}{\theta}e^{-\theta t}\int_s^te^{\theta v}(v-s)^{2H-2}dv\\
&:=&J^1_t+J_t^2.
\end{eqnarray*}
Since $H<1$, $J^1_t=\frac{H}{\theta}[t^{2H-1}-(t-s)^{2H-1}]\rightarrow 0 $ as $t\rightarrow \infty$.\\
On the other hand,
\begin{eqnarray*}J_t^2&=&\frac{H(2H-1)}{\theta}e^{-\theta t}\int_s^te^{\theta v}(v-s)^{2H-2}dv\\
\\&=&\frac{H(2H-1)e^{\theta s}}{\theta}e^{-\theta t}\int_0^{t-s}e^{\theta u}u^{2H-2}du\\
&\leq&\frac{H(2H-1)e^{\theta s}}{\theta}e^{-\theta t}\int_0^{t}e^{\theta u}u^{2H-2}du\\
&=&\frac{H(2H-1)e^{\theta s}}{\theta}\int_0^{t}e^{-\theta v}(t-v)^{2H-2}dv\\
&=&\frac{H(2H-1)e^{\theta s}}{\theta}t^{2H-1}\int_0^{1}e^{-\theta tu}(1-u)^{2H-2}du
\end{eqnarray*}Fix $u\in(0,1)$. The function $t\in[0,\infty)\mapsto t^{2H-1}e^{-\theta tu}$ attains its maximum at $t=\frac{2H-1}{\theta u}$. Then
\[\sup_{t\geq0}(t^{2H-1}e^{-\theta tu})=ce^{-\frac{2H-1}{u}}u^{1-2H}\leq c u^{1-2H},\] with $c=\left(\frac{2H-1}{\theta}\right)^{2H-1}$. In addition, $\int_0^{1}u^{1-2H}(1-u)^{2H-2}du<\infty$, and for any $u\in(0,1)$,
\[t^{2H-1}e^{-\theta tu}(1-u)^{2H-2}\rightarrow 0 \mbox{ as } t\rightarrow \infty.\] Therefore, using the dominated convergence theorem, we obtain that $J^2_t$ converges to $0$ as $t\rightarrow \infty$.
\\Thus, we deduce the desired conclusion.
\end{proof}\\
\begin{proof}[Proof of Lemma \ref{lemma in l^2}]
Let us prove the convergence (\ref{1}). We have
\begin{eqnarray*}&&e^{-\theta t}E\left(\int_0^t\delta B_se^{-\theta s}\int_0^s\delta B_re^{\theta r}\right)^2\\
&=&e^{-\theta t}E\left(\int_0^t\int_0^se^{-\theta |s-r|}\delta B_r\delta B_s\right)^2\\&=&e^{-\theta t}
E\left(\frac{1}{2}I_2( e^{-\theta |s-r|}1_{[0,t]^2} )\right)^2
\\&=&\frac{H^2(2H-1)^2}{2}e^{-\theta t}
\int_{[0,t]^4} e^{-\theta |v-s|}e^{-\theta |u-r|} |v-u|^{2H-2}|s-r|^{2H-2}dudv drds\\&\leq&\frac{H^2(2H-1)^2}{2}e^{-\theta t}
\int_{[0,t]^4} |v-u|^{2H-2}|s-r|^{2H-2}dudv drds\\&=& \frac{1}{2}[E(B_t^2)]^2e^{-\theta t}=\frac{1}{2} t^{4H}e^{-\theta t}
\\&\rightarrow& 0 \ \mbox{ as }   t\rightarrow \infty.
\end{eqnarray*}
For the convergence (\ref{2}), we have
\begin{eqnarray*}&&H(2H-1)e^{-\frac{\theta t}{2}}\int_0^tdse^{-\theta s}\int_0^sdre^{\theta r}|s-r|^{2H-2}\\&\leq&
H(2H-1)e^{-\frac{\theta t}{2}}\int_0^tds \int_0^sdr |s-r|^{2H-2}\\&=&\frac{ t^{2H}}{2}e^{-\frac{\theta t}{2}}
\\&\rightarrow& 0 \ \mbox{ as }   t\rightarrow \infty.
\end{eqnarray*}This finishes the proof.
\end{proof}\\
\begin{proof}[Proof of the theorem \ref{convergence distribution}]
 By combining (\ref{representation of LSE}) and Lemma \ref{decomposition Young Skorohod}, we can write,
\begin{eqnarray*} e^{\theta t}\left(\widehat{\theta}_t-\theta\right)&=&\frac{e^{\theta t} \int_0^tdB_se^{\theta s}\int_0^sdB_re^{-\theta r}}{\int_0^te^{2\theta s}\xi_s^2ds}\\&=&\frac{\xi_t\xi_{\infty}}{e^{-2\theta t}\int_0^te^{2\theta s}\xi_s^2ds}\times\frac{e^{-\theta t}\int_0^te^{\theta s}dB_s}{\xi_{\infty}}\\&& -
\frac{e^{-\theta t}\int_0^t\delta B_se^{-\theta s}\int_0^s\delta B_re^{\theta r}}{e^{-2\theta t}\int_0^te^{2\theta s}\xi_s^2ds}\\&&-H(2H-1)\frac{e^{-\theta t}\int_0^tdse^{-\theta s}\int_0^sdre^{\theta r}|s-r|^{2H-2}}{e^{-2\theta t}\int_0^te^{2\theta s}\xi_s^2ds}\\&:=&A_t^{\theta}\times B_t^{\theta}-C_t^{\theta}- D_t^{\theta}.
\end{eqnarray*}
Using Lemme \ref{convergence p.s of xi} and Lemma \ref{convergence of dominating}, we obtain that
\begin{eqnarray*}A_t^{\theta}\longrightarrow2\theta \mbox{ almost surely    as } t\longrightarrow\infty.
\end{eqnarray*}
According to Lemma \ref{lemma in law}, we deduce
\begin{eqnarray*}B_t^{\theta} &\overset{\texttt{law}}{\longrightarrow}&\frac{\sqrt{H\Gamma(2H)}}{\theta^{H}}\frac{N}{\xi_{\infty}}
\ \mbox{ as } t\longrightarrow\infty.
\end{eqnarray*}Moreover, \[\frac{\sqrt{H\Gamma(2H)}}{\theta^{H}}\frac{N}{\xi_{\infty}}\overset{\texttt{law}}{=}\mathcal{C}(1),\] because
$\frac{\theta^{H}\xi_{\infty}}{\sqrt{H\Gamma(2H)}}\backsim \mathcal{N}(0,1)$  and  $N\sim\mathcal{N}(0,1)$ are independent.\\ Thus, by Slutsky's theorem, we conclude that
\[A_t^{\theta}\times B_t^{\theta}\overset{\texttt{law}}{\longrightarrow}2\theta \mathcal{C}(1) \ \mbox{ as } t\longrightarrow\infty.\]
On the other hand, it follows from  Lemma \ref{convergence of dominating} and Lemma \ref{lemma in l^2}, that
\begin{eqnarray*}C_t^{\theta}\overset{\texttt{prob.}}{\longrightarrow}0 \   \mbox{ as } t\longrightarrow\infty,
\end{eqnarray*}
and
\begin{eqnarray*}D_t^{\theta}\longrightarrow0\ \mbox{ almost surely }  \   \mbox{ as } t\longrightarrow\infty.
\end{eqnarray*}
Finally, by combining the previous convergences, the proof of Theorem \ref{convergence distribution} is done.\end{proof}
\\
\vspace{3mm}\\
{\bf Acknowledgments.} The authors would like to thank Ivan Nourdin for many valuable discussions on the
subject. We warmly thank him for proving that  $J_t$ converges to zero (see Proof of Lemma~~\ref{lemma in law}).

\bibliographystyle{amsplain}

\end{document}